\documentclass[reqno,a4paper]{amsart}
\usepackage{amssymb}
\usepackage[colorlinks=true,hypertex]{hyperref}
\allowdisplaybreaks[4]
\theoremstyle{plain}
\newtheorem{thm}{Theorem}

\newtheorem{cor}{Corollary}
\theoremstyle{remark}
\newtheorem{rem}{Remark}
\DeclareMathOperator{\td}{d\mspace{-2mu}}

\date{Drafted on 24 November 2008 and completed on 26 November 2008 in VU's Student Village}
\date{}

\begin{document}

\title{Sharp inequalities for polygamma functions}

\author[F. Qi]{Feng Qi}
\address[F. Qi]{Research Institute of Mathematical Inequality Theory, Henan Polytechnic University, Jiaozuo City, Henan Province, 454010, China}
\email{\href{mailto: F. Qi <qifeng618@gmail.com>}{qifeng618@gmail.com}, \href{mailto: F. Qi <qifeng618@hotmail.com>}{qifeng618@hotmail.com}, \href{mailto: F. Qi <qifeng618@qq.com>}{qifeng618@qq.com}}
\urladdr{\url{http://qifeng618.spaces.live.com}}

\author[B.-N. Guo]{Bai-Ni Guo}
\address[B.-N. Guo]{School of Mathematics and Informatics, Henan Polytechnic University, Jiaozuo City, Henan Province, 454010, China}
\email{\href{mailto: B.-N. Guo <bai.ni.guo@gmail.com>}{bai.ni.guo@gmail.com}, \href{mailto: B.-N. Guo <bai.ni.guo@hotmail.com>}{bai.ni.guo@hotmail.com}}
\urladdr{\url{http://guobaini.spaces.live.com}}

\begin{abstract}
The main aim of this paper is to prove that the double inequality
\begin{equation*}
\frac{(k-1)!}{\Bigl\{x+\Bigl[\frac{(k-1)!}{\vert\psi^{(k)}(1)\vert}\Bigr]^{1/k}\Bigr\}^k} +\frac{k!}{x^{k+1}}<\bigl\vert\psi^{(k)}(x)\bigr\vert<\frac{(k-1)!}{\bigl(x+\frac12\bigr)^k}+\frac{k!}{x^{k+1}}
\end{equation*}
holds for $x>0$ and $k\in\mathbb{N}$ and that the constants $\Bigl[\frac{(k-1)!}{\vert\psi^{(k)}(1)\vert}\Bigr]^{1/k}$ and $\frac12$ are the best possible. In passing, some related inequalities and (logarithmically) complete monotonicity results concerning the gamma, psi and polygamma functions are surveyed.
\end{abstract}



\thanks{This paper was typeset using \AmS-\LaTeX}

\maketitle

\section{Introduction}
\subsection{Completely monotonic functions}
Recall~\cite[Chapter~XIII]{mpf-93} and~\cite[Chapter~IV]{widder} that a function $f(x)$ is said to be completely monotonic on an interval $I\subseteq\mathbb{R}$ if $f(x)$ has derivatives of all orders on $I$ and
\begin{equation}
0\le(-1)^{k}f^{(k)}(x)<\infty
\end{equation}
holds for all $k\geq0$ on $I$.
\par
The celebrated Bernstein-Widder Theorem~\cite[p.~161]{widder} states that a function $f(x)$ is completely monotonic on $(0,\infty)$ if and only if
\begin{equation}\label{converge}
f(x)=\int_0^\infty e^{-xs}\td\mu(s),
\end{equation}
where $\mu$ is a nonnegative measure on $[0,\infty)$ such that the integral~\eqref{converge} converges for all $x>0$. This means that a function $f(x)$ is completely monotonic on $(0,\infty)$ if and only if it is a Laplace transform of the measure $\mu$.
\par
The completely monotonic functions have applications in different branches of mathematical sciences. For example, they play some role in combinatorics, numerical and asymptotic analysis, physics, potential theory, and probability theory.
\par
The most important properties of completely monotonic functions can be found in~\cite[Chapter~XIII]{mpf-93}, \cite[Chapter~IV]{widder} and closely-related references therein.

\subsection{Logarithmically completely monotonic functions}
Recall also~\cite{Atanassov, minus-one} that a function $f$ is said to be logarithmically completely monotonic on an interval $I\subseteq\mathbb{R}$ if it has derivatives of all orders on $I$ and its logarithm $\ln f$ satisfies
\begin{equation}\label{lcm-dfn}
0\le(-1)^k[\ln f(x)]^{(k)}<\infty
\end{equation}
for $k\in\mathbb{N}$ on $I$.
\par
By looking through ``logarithmically completely monotonic function'' in the database \href{http://www.ams.org/mathscinet/}{MathSciNet}, it is found that this phrase was first used in~\cite{Atanassov}, but with no a word to explicitly define it. Thereafter, it seems to have been ignored by the mathematical community. In early 2004, this terminology was recovered in~\cite{minus-one} and it was immediately referenced in~\cite{auscm-rgmia}, the preprint of the paper~\cite{e-gam-rat-comp-mon}. A natural question that one may ask is: Whether is this notion trivial or not? In~\cite[Theorem~4]{minus-one}, it was proved that all logarithmically completely monotonic functions are also completely monotonic, but not conversely. This result was formally published when revising~\cite{compmon2}. Hereafter, this conclusion and its proofs were dug in~\cite{CBerg, Gao-0709.1126v2-Arxiv, grin-ismail} and~\cite{schur-complete} (the preprint of~\cite{absolute-mon.tex}) once and again. Furthermore,  in the paper~\cite{CBerg}, the logarithmically completely monotonic functions on $(0,\infty)$ were characterized as the infinitely divisible completely monotonic functions studied in~\cite{horn} and all Stieltjes transforms were proved to be logarithmically completely monotonic on $(0,\infty)$. For more information, please refer to~\cite{CBerg}.

\subsection{The gamma and polygamma functions}
It is well-known that the classical Euler gamma function $\Gamma(x)$ may be defined for $x>0$ by
\begin{equation}\label{egamma}
\Gamma(x)=\int^\infty_0t^{x-1} e^{-t}\td t.
\end{equation}
The logarithmic derivative of $\Gamma(x)$, denoted by $\psi(x)=\frac{\Gamma'(x)}{\Gamma(x)}$, is called the psi or digamma function, and $\psi^{(k)}(x)$ for $k\in\mathbb{N}$ are called the polygamma functions. It is common knowledge that these functions are fundamental and important and that they have much extensive applications in mathematical sciences.

\subsection{The first kind of inequalities for the psi and polygamma functions}
In~\cite[Theorem~2.1]{Ismail-Muldoon-119}, \cite[Lemma~1.3]{sandor-gamma-2-ITSF.tex} and~\cite[Lemma~3]{sandor-gamma-2-ITSF.tex-rgmia}, the function $\psi(x)-\ln x+\frac{\alpha}x$ was proved to be completely monotonic on $(0,\infty)$ if and only if $\alpha\ge1$, so is its negative if and only if $\alpha\le\frac12$. In~\cite[Theorem~2]{chen-qi-log-jmaa} and~\cite[Theorem~2.1]{Muldoon-78}, the function $\frac{e^x\Gamma(x)} {x^{x-\alpha}}$ was proved to be logarithmically completely monotonic on $(0,\infty)$ if and only if $\alpha\ge1$, so is its reciprocal if and only if $\alpha\le\frac12$. From these, the following double inequalities were derived and employed in~\cite{subadditive-qi.tex, Comp-Mon-Digamma-Trigamma-Divided.tex, simple-equiv-simple-rev.tex, Extension-TJM-2003.tex, theta-new-proof.tex, polygamma-square-polygamma.tex, Open-TJM-2003.tex, AAM-Qi-09-PolyGamma.tex, property-psi.tex, subadditive-qi-guo.tex, subadditive-qi-3.tex}: For $x\in(0,\infty)$ and $k\in\mathbb{N}$, we have
\begin{equation}\label{qi-psi-ineq-1}
\ln x-\frac1x<\psi(x)<\ln x-\frac1{2x}
\end{equation}
and
\begin{equation}\label{qi-psi-ineq}
\frac{(k-1)!}{x^k}+\frac{k!}{2x^{k+1}} <\bigl\vert\psi^{(k)}(x)\bigr\vert<\frac{(k-1)!}{x^k}+\frac{k!}{x^{k+1}}.
\end{equation}
\par
In~\cite[Theorem~9]{psi-alzer}, it was proved that if $k\ge1$ and $n\ge0$ are integers then
\begin{equation}\label{pai-alzer-thm9-ineq}
S_k(2n;x)<\bigl\vert\psi^{(k)}(x)\bigr\vert<S_k(2n+1;x)
\end{equation}
holds for $x>0$, where
\begin{equation}
S_k(p;x)=\frac{(k-1)!}{x^k}+\frac{k!}{2x^{k+1}} +\sum_{i=1}^pB_{2i}\Biggl[\prod_{j=1}^{k-1}(2i+j)\Biggr]\frac1{x^{2i+k}}
\end{equation}
with the usual convention that an empty sum is nil and $B_i$ for $i\ge0$ are Bernoulli numbers defined by
\begin{equation}
\frac{t}{e^t-1}=\sum_{i=0}^\infty B_i\frac{t^i}{i!} =1-\frac{x}2+\sum_{j=1}^\infty B_{2j}\frac{x^{2j}}{(2j)!}, \quad\vert x\vert<2\pi.
\end{equation}
\par
In~\cite{Allasia-Gior-Pecaric-MIA-02}, among other things, the following double inequalities were procured: For $x>\frac12$, we have
\begin{equation}\label{Allasia-Gior-Pecaric-ineq-1}
\sum_{k=1}^{2N+1}\frac{B_{2k}\bigl(\frac12\bigr)}{2k\bigl(x-\frac12\bigr)^{2k}} <\ln\biggl(x-\frac12\biggr)-\psi(x) <\sum_{k=1}^{2N}\frac{B_{2k}\bigl(\frac12\bigr)}{2k\bigl(x-\frac12\bigr)^{2k}}
\end{equation}
and
\begin{multline}\label{Allasia-Gior-Pecaric-ineq-2}
\frac{(n-1)!}{\bigl(x-\frac12\bigr)^n} +\sum_{k=1}^{2N+1}\frac{(n+2k-1)!B_{2k}\bigl(\frac12\bigr)}{(2k)!\bigl(x-\frac12\bigr)^{n+2k}} <\bigl\vert\psi^{(n)}(x)\bigr\vert\\
<\frac{(n-1)!}{\bigl(x-\frac12\bigr)^n} +\sum_{k=1}^{2N}\frac{(n+2k-1)!B_{2k}\bigl(\frac12\bigr)}{(2k)!\bigl(x-\frac12\bigr)^{n+2k}},
\end{multline}
where $n\ge1$, $N\ge0$, an empty sum is understood to be nil, and
\begin{equation}
B_k\biggl(\frac12\biggr)=\biggl(\frac1{2^{k-1}}-1\biggr)B_k,\quad k\ge0.
\end{equation}
When replacing $2N$ by $2N-1$, inequalities~\eqref{Allasia-Gior-Pecaric-ineq-1} and~\eqref{Allasia-Gior-Pecaric-ineq-2} are reversed. In particular, for $n=1$ and $N=0$,
\begin{equation}\label{Allasia-Gior-Pecaric-ineq-n=1N=0}
\frac1{x-\frac12}-\frac1{12\bigl(x-\frac12\bigr)^2}<\psi'(x)<\frac1{x-\frac12},\quad x>\frac12.
\end{equation}
\par
It is obvious that if taking $x\to\bigl(\frac12\bigr)^+$ the lower and upper bounds in~\eqref{Allasia-Gior-Pecaric-ineq-2} tend to $-\infty$ and $\infty$ respectively, but the middle term tends to a limited constant. This implies that inequalities in~\eqref{Allasia-Gior-Pecaric-ineq-1} and~\eqref{Allasia-Gior-Pecaric-ineq-2}, including~\eqref{Allasia-Gior-Pecaric-ineq-n=1N=0}, may be not ideal.
\par
It is noted that the inequality~\eqref{pai-alzer-thm9-ineq} was deduced from~\cite[Theorem~8]{psi-alzer} which states that the functions
\begin{equation}\label{fn(x)}
F_n(x)=\ln\Gamma(x)-\biggl(x-\frac12\biggr)\ln x+x-\frac12\ln(2\pi) -\frac12\sum_{j=1}^{2n}\frac{B_{2j}}{2j(2j-1)x^{2j-1}}
\end{equation}
and
\begin{equation}
G_n(x)=-\ln\Gamma(x)+\biggl(x-\frac12\biggr)\ln x-x+\frac12\ln(2\pi) +\frac12\sum_{j=1}^{2n+1}\frac{B_{2j}}{2j(2j-1)x^{2j-1}}\label{gn(x)}
\end{equation}
are completely monotonic on $(0,\infty)$.
\par
In~\cite[Theorem~1]{merkle-jmaa-96}, the convexity of the functions $F_n(x)$ and $G_n(x)$ were presented alternatively.
\par
Stimulated by~\cite{poly-comp-mon.tex}, the complete monotonicity of $F_n(x)$ and $G_n(x)$ were simply verified in~\cite[Theorem~2]{Koumandos-jmaa-06} again.

\subsection{The second kind of inequalities for the psi and polygamma functions}
In~\cite[Theorem~1]{Guo-Qi-Srivastava2007.tex}, the function
\begin{equation}\label{g-{alpha,beta}(x)}
g_{\alpha,\beta}(x)=\biggl[\frac{e^x\Gamma(x+1)} {(x+\beta)^{x+\beta}}\biggr]^\alpha
\end{equation}
for real numbers $\alpha\ne 0$ and $\beta$ was shown to be logarithmically completely monotonic with respect to $x\in(\max\{0,-\beta\},\infty)$ if and only if either $\alpha>0$ and $\beta\geq1$ or $\alpha<0$ and $\beta\leq\frac12$. As a result, the following double inequalities~\eqref{qi-psi-ineq-beta-1} and~\eqref{qi-psi-ineq-beta-2} were deduced in~\cite[Lemma~2]{Open-TJM-2003.tex} and used in~\cite[Lemma~3]{Open-TJM-2003.tex}: For $x\in(0,\infty)$ and $k\in\mathbb{N}$, we have
\begin{equation}\label{qi-psi-ineq-beta-1}
\ln\biggr(x+\frac12\biggl)-\frac1x<\psi(x)<\ln(x+1)-\frac1x
\end{equation}
and
\begin{equation}\label{qi-psi-ineq-beta-2}
\frac{(k-1)!}{(x+1)^k}+\frac{k!}{x^{k+1}}<\bigl\vert\psi^{(k)}(x)\bigr\vert <\frac{(k-1)!}{\bigl(x+\frac12\bigr)^k}+\frac{k!}{x^{k+1}}.
\end{equation}
\par
It is clear that the left-hand side inequality in~\eqref{qi-psi-ineq-beta-1} and the right-hand side inequality in~\eqref{qi-psi-ineq-beta-2} are better than the left-hand side inequality in~\eqref{qi-psi-ineq-1} and the right-hand side inequality in~\eqref{qi-psi-ineq}. It is also easy to see that the right-hand side inequality in~\eqref{qi-psi-ineq-beta-1} and the left-hand side inequality in~\eqref{qi-psi-ineq-beta-2} are more exact than the right-hand side inequality in~\eqref{qi-psi-ineq-1} and the left-hand side inequality in~\eqref{qi-psi-ineq} when $x>0$ is close enough to $0$, but not when $x>0$ is large enough.
\par
For more information on further investigation of functions similar to~\eqref{g-{alpha,beta}(x)}, please refer to the research papers~\cite{Guo-Qi-12-07.tex, Guo-Qi-Srivastava2007-02.tex, Guo-Srivastava-AML-08}, the expository article~\cite{bounds-two-gammas.tex} and related references therein.

\subsection{A sharp inequality for the psi function and related results}
In~\cite[Lemma~1.7]{Batir-Arch-Math-08} and~\cite[Theorem~1]{Infinite-family-Digamma.tex}, it was proved that the double inequality
\begin{equation}\label{corollary2.3-rew}
    \ln\biggl(x+\frac12\biggr)-\frac1x< \psi(x)< \ln(x+e^{-\gamma})-\frac1x
\end{equation}
holds on $(0,\infty)$ and the scalars $\frac12$ and $e^{-\gamma}=0.56\dotsm$ in~\eqref{corollary2.3-rew} are the best possible.
\par
It is clear that the inequality~\eqref{corollary2.3-rew} refines and sharpens~\eqref{qi-psi-ineq-beta-1}. The inequality~\eqref{corollary2.3-rew} has relations with~\eqref{qi-psi-ineq-1} as~\eqref{qi-psi-ineq-beta-1} does.
\par
More strongly, the function
\begin{equation}\label{Q(x)-dfn-2}
Q(x)=e^{\psi(x+1)}-x
\end{equation}
was proved in~\cite[Theorem~2]{Infinite-family-Digamma.tex} to be strictly decreasing and convex on $(-1,\infty)$ with $\lim_{x\to\infty}Q(x)=\frac12$. The basic tools of the proofs in~\cite{Infinite-family-Digamma.tex} include
\begin{equation}\label{batir-alzer-ineq}
\psi'(x)e^{\psi(x)}<1,\quad x>0
\end{equation}
and
\begin{equation}\label{positivity}
[\psi'(x)]^2+\psi''(x)>0,\quad x>0.
\end{equation}
\par
Among other things, the monotonicity and convexity of the function~\eqref{Q(x)-dfn-2} were also derived in~\cite[Corollary~2 and Corollary~3]{egp}: For all $t>0$, the function $\exp\{\psi(x+t)\}-x$ is decreasing with respect to $x\in[0,\infty)$; For all $t>0$, the digamma function can be written in a way:
\begin{equation*}
\psi(x+t)=\ln(x+\delta(x)),\quad x>0,
\end{equation*}
where $\delta$ is decreasing convex function which maps $[0,\infty)$ onto $\bigl[e^{\psi(t)},t-\frac12\bigr)$.
\par
The one-sided inequality~\eqref{batir-alzer-ineq} was deduced in~\cite[Corollary~2]{egp} and recovered in~\cite[Lemma~1.1]{batir-new} and~\cite[Lemma~1.1]{batir-new-RGMIA}.
\par
The sharp double inequality~\eqref{corollary2.3-rew} is the special case $t=1$ of the following double inequality obtained in~\cite[Corollary~3]{egp}: For all $x>0$ and $t>0$, it holds that
\begin{equation}\label{egp-ineq-c3}
\ln\biggl(x+\frac{2t-1}2\biggr)<\psi(x+t)<\ln(x+\exp(\psi(t))).
\end{equation}
\par
It is worthwhile to remark that the left-hand side inequality in~\eqref{egp-ineq-c3} for $x+t\le\frac12$ is meaningless.
\par
Replacing $x$ by $x+t$ in~\eqref{corollary2.3-rew} yields
\begin{equation}\label{corollary2.3-rew-x+t}
    \ln\biggl(x+t+\frac12\biggr)-\frac1{x+t}< \psi(x+t)< \ln(x+t+e^{-\gamma})-\frac1{x+t}
\end{equation}
for all $x>0$ and $t>0$. The left-hand side inequality in~\eqref{corollary2.3-rew-x+t} extends and refines the corresponding one in~\eqref{egp-ineq-c3} and their right-hand side inequalities do not contain each other.
\par
For information about the history and backgrounds of the function~\eqref{Q(x)-dfn-2} and inequalities~\eqref{batir-alzer-ineq} and~\eqref{egp-ineq-c3}, please refer to the expository papers~\cite{bounds-two-gammas.tex, Wendel2Elezovic.tex} and lots of references therein.
\par
The inequality~\eqref{positivity} was first obtained in the proof of~\cite[p.~208, Theorem~4.8]{forum-alzer} and recovered in~\cite[Theorem~2.1]{batir-interest}, \cite[Lemma~1.1]{batir-new} and~\cite[Lemma~1.1]{batir-new-RGMIA}.
\par
In~\cite[Remark~1.3]{batir-jmaa-06-05-065}, it was pointed out that the inequality~\eqref{positivity} is the special case $n=1$ of~\cite[Lemma~1.2]{batir-jmaa-06-05-065} which reads
\begin{equation}
(-1)^n\psi^{(n+1)}(x)<\frac{n}{\sqrt[n]{(n-1)!}\,}\bigl[(-1)^{n-1}\psi^{(n)}(x)\bigr]^{1+1/n}
\end{equation}
for $x>0$ and $n\in\mathbb{N}$. This inequality can be restated more meaningfully as
\begin{equation}\label{gen-1-1}
\sqrt[n+1]{\frac{\bigl|\psi^{(n+1)}(x)\bigr|}{n!}}\, <\sqrt[n]{\frac{\bigl\vert\psi^{(n)}(x)\bigr\vert }{(n-1)!}}.
\end{equation}
\par
In~\cite[Lemma~4.6]{alzer-grinshpan}, the inequality~\eqref{positivity} was generalized to the $q$-analogue.
\par
In~\cite{notes-best-simple-open, notes-best-simple.tex-rgmia}, the preprints of~\cite{notes-best-simple-open.tex-rev, notes-best-simple-rev.tex}, the divided difference
\begin{equation}\label{Delta-dfn}
\Delta_{s,t}(x)=\begin{cases}\bigg[\dfrac{\psi(x+t) -\psi(x+s)}{t-s}\bigg]^2
+\dfrac{\psi'(x+t)-\psi'(x+s)}{t-s},&s\ne t\\
[\psi'(x+s)]^2+\psi''(x+s),&s=t
\end{cases}
\end{equation}
for $|t-s|<1$ and $-\Delta_{s,t}(x)$ for $|t-s|>1$ were proved to be completely monotonic with respect to $x\in(-\min\{s,t\},\infty)$. In particular, the function $[\psi'(x)]^2+\psi''(x)$ appearing in~~\eqref{positivity} is completely monotonic on $(0,\infty)$.
\par
For $m,n\in\mathbb{N}$, let
\begin{equation}
f_{m,n}(x)=\psi^{(n)}(x)+\bigr[\psi^{(m)}(x)\bigl]^2, \quad x>0.
\end{equation}
In~\cite{polygamma-square-polygamma.tex}, it was revealed that the functions $f_{1,2}(x)$ and $f_{m,2n-1}(x)$ are completely monotonic on $(0,\infty)$, but the functions $f_{m,2n}(x)$ for $(m,n)\ne(1,1)$ are not monotonic and does not keep the same sign on $(0,\infty)$. This means that $f_{1,2}(x)$ is the only nontrivial completely monotonic function on $(0,\infty)$ among all functions $f_{m,n}(x)$ for $m,n\in\mathbb{N}$.
\par
In~\cite{AAM-Qi-09-PolyGamma.tex}, the function
\begin{equation}\label{di-tetra-gamma-lambda}
\Delta_{\lambda}(x)=[\psi'(x)]^2+\lambda\psi''(x)
\end{equation}
was shown to be completely monotonic on $(0,\infty)$ if and only if $\lambda\le1$.
\par
For real numbers $s$, $t$, $\alpha=\min\{s,t\}$ and $\lambda$, define
\begin{equation}\label{Delta-lambda-dfn}
\Delta_{s,t;\lambda}(x)=\begin{cases}\bigg[\dfrac{\psi(x+t) -\psi(x+s)}{t-s}\bigg]^2
+\lambda\dfrac{\psi'(x+t)-\psi'(x+s)}{t-s},&s\ne t\\
[\psi'(x+s)]^2+\lambda\psi''(x+s),&s=t
\end{cases}
\end{equation}
with respect to $x\in(-\alpha,\infty)$. In~\cite{Comp-Mon-Digamma-Trigamma-Divided.tex}, the following complete monotonicity were established:
\begin{enumerate}
\item
For $0<|t-s|<1$,
\begin{enumerate}
\item
the function $\Delta_{s,t;\lambda}(x)$ is completely monotonic on $(-\alpha,\infty)$ if and only if $\lambda\le1$,
\item
so is the function $-\Delta_{s,t;\lambda}(x)$ if and only if $\lambda\ge\frac1{|t-s|}$;
\end{enumerate}
\item
For $|t-s|>1$,
\begin{enumerate}
\item
the function $\Delta_{s,t;\lambda}(x)$ is completely monotonic on $(-\alpha,\infty)$ if and only if $\lambda\le\frac1{|t-s|}$,
\item
so is the function $-\Delta_{s,t;\lambda}(x)$ if and only if $\lambda\ge1$;
\end{enumerate}
\item
For $s=t$, the function $\Delta_{s,s;\lambda}(x)$ is completely monotonic on $(-s,\infty)$ if and only if $\lambda\le1$;
\item
For $|t-s|=1$,
\begin{enumerate}
\item
the function $\Delta_{s,t;\lambda}(x)$ is completely monotonic if and only if $\lambda<1$,
\item
so is the function $-\Delta_{s,t;\lambda}(x)$ if and only if $\lambda>1$,
\item
and $\Delta_{s,t;1}(x)\equiv0$.
\end{enumerate}
\end{enumerate}
These results generalize the claim in the proof of~\cite{Kazarinoff-56}. For detailed information, see related texts remarked in the expository article~\cite{Wendel-Gautschi-type-ineq.tex}.
\par
In~\cite[Remark~2.3]{batir-jmaa-06-05-065}, it was pointed out that the inequality
\begin{equation}\label{psi(x+frac12}
\psi''(x)+\biggl[\psi'\biggl(x+\frac12\biggr)\biggr]^2<0
\end{equation}
for $x>0$ is a direct consequence of~\cite[Theorem~2.2]{batir-jmaa-06-05-065}: For $x>0$, $1\le k\le n-1$ and $n\in\mathbb{N}$, we have
\begin{multline}
(n-1)!\biggl[\frac{\psi^{(k)}\bigl(x+\frac12\bigr)}{(-1)^{k-1}(k-1)!}\biggr]^{n/k} <(-1)^{n+1}\psi^{(n)}(x)\\
<(n-1)!\biggl[\frac{\psi^{(k)}(x)}{(-1)^{k-1}(k-1)!}\biggr]^{n/k}
\end{multline}
which can be rewritten as
\begin{equation}\label{gen-2-1}
\sqrt[k]{\frac{\bigl|\psi^{(k)}\bigl(x+\frac12\bigr)\bigr|}{(k-1)!}}\, <\sqrt[n]{\frac{\bigl\vert\psi^{(n)}(x)\bigr\vert}{(n-1)!}}\, <\sqrt[k]{\frac{\bigl|\psi^{(k)}(x)\bigr|}{(k-1)!}}.
\end{equation}
\par
In~\cite{AAM-Qi-09-PolyGamma.tex}, the inequality~\eqref{psi(x+frac12} was generalized to the complete monotonicity: For real number $\alpha\in\mathbb{R}$ and $x>-\min\{0,\alpha\}$,
\begin{enumerate}
\item
the function $\psi''(x)+[\psi'(x+\alpha)]^2$ is completely monotonic if and only if $\alpha\le0$;
\item
the function $-\bigl\{\psi''(x)+[\psi'(x+\alpha)]^2\bigr\}$ is completely monotonic if
\begin{equation}\label{alpha-g-1}
\alpha\ge\sup_{x\in(0,\infty)}\frac{x}{\phi^{-1} \left([2(x+1)^2-1]e^{2x}\right)},
\end{equation}
where $\phi^{-1}$ denotes the inverse function of $\phi(x)=x\coth x$ on $(0,\infty)$.
\end{enumerate}
\par
In passing, it is noted that the results demonstrated in~\cite{notes-best-simple-equiv.tex-RGMIA, notes-best-simple-equiv.tex, simple-equiv.tex, simple-equiv-simple-rev.tex} have very close relations with the above mentioned conclusions.

\subsection{Main results of this paper}
The main aim of this paper is to sharpen the double inequality~\eqref{qi-psi-ineq-beta-2} and to generalize the sharp inequality~\eqref{corollary2.3-rew} to the cases for polygamma functions.
\par
The main result of this paper may be stated as the following theorem.

\begin{thm}\label{sharp-ineq-polygamma-thm}
For $x>0$ and $k\in\mathbb{N}$, the double inequality
\begin{equation}\label{poly-alpha-beta-best-ineq}
\frac{(k-1)!}{\Bigl\{x+\Bigl[\frac{(k-1)!}{\vert\psi^{(k)}(1)\vert}\Bigr]^{1/k}\Bigr\}^k} +\frac{k!}{x^{k+1}}<\bigl\vert\psi^{(k)}(x)\bigr\vert<\frac{(k-1)!}{\bigl(x+\frac12\bigr)^k}+\frac{k!}{x^{k+1}}
\end{equation}
holds and the constants $\Bigl[\frac{(k-1)!}{\vert\psi^{(k)}(1)\vert}\Bigr]^{1/k}$ and $\frac12$ in~\eqref{poly-alpha-beta-best-ineq} are the best possible.
\end{thm}

As direct consequences of Theorem~\ref{sharp-ineq-polygamma-thm}, the following corollaries may be derived.

\begin{cor}\label{sharp-ineq-polygamma-cor-1}
For $x>0$ and $k\in\mathbb{N}$, the double inequality
\begin{equation}\label{poly-alpha-beta-best-ineq-c-1}
\frac{(k-1)!}{\Bigl\{x+\Bigl[\frac{(k-1)!}{\vert\psi^{(k)}(1)\vert}\Bigr]^{1/k}\Bigr\}^k} <\bigl\vert\psi^{(k)}(x+1)\bigr\vert<\frac{(k-1)!}{\bigl(x+\frac12\bigr)^k}
\end{equation}
is valid and the scalars $\Bigl[\frac{(k-1)!}{\vert\psi^{(k)}(1)\vert}\Bigr]^{1/k}$ and $\frac12$ in~\eqref{poly-alpha-beta-best-ineq-c-1} are the best possible.
\end{cor}

\begin{cor}\label{sharp-ineq-polygamma-cor-2}
Under the usual convention that an empty sum is understood to be nil, the double inequalities
\begin{multline}\label{poly-alpha-beta-best-ineq-c-3}
k!\sum_{i=1}^{m}\frac1{(x+i-1)^{k+1}}+ \frac{(k-1)!}{\Bigl\{x+m-1+\Bigl[\frac{(k-1)!}{\vert\psi^{(k)}(1)\vert}\Bigr]^{1/k}\Bigr\}^k}
<\bigl\vert\psi^{(k)}(x)\bigr\vert\\
<k!\sum_{i=1}^{m}\frac1{(x+i-1)^{k+1}} +\frac{(k-1)!}{\bigl(x+m-\frac12\bigr)^k}
\end{multline}
and
\begin{multline}\label{poly-alpha-beta-best-ineq-c-2}
\frac{(k-1)!}{\Bigl\{x+\Bigl[\frac{(k-1)!}{\vert\psi^{(k)}(1)\vert}\Bigr]^{1/k}\Bigr\}^k} -k!\sum_{i=1}^{m-1}\frac1{(x+i)^{k+1}}
<\bigl\vert\psi^{(k)}(x+m)\bigr\vert\\
<\frac{(k-1)!}{\bigl(x+\frac12\bigr)^k}-k!\sum_{i=1}^{m-1}\frac1{(x+i)^{k+1}}
\end{multline}
hold for $x>0$ and $k,m\in\mathbb{N}$. Meanwhile, the quantities $\Bigl[\frac{(k-1)!}{\vert\psi^{(k)}(1)\vert}\Bigr]^{1/k}$ and $\frac12$ in inequalities~\eqref{poly-alpha-beta-best-ineq-c-3} and~\eqref{poly-alpha-beta-best-ineq-c-2} are the best possible.
\end{cor}

\begin{rem}
When approximating the psi function $\psi(x)$ and polygamma functions $\psi^{(k)}(x)$ for $k\in\mathbb{N}$, the double inequalities~\eqref{corollary2.3-rew} and~\eqref{poly-alpha-beta-best-ineq} are more accurate than~\eqref{qi-psi-ineq-1} and~\eqref{qi-psi-ineq} as long as $x$ is enough close to $0$ from the right-hand side. For example, the right-hand side inequality in~\eqref{qi-psi-ineq-beta-2} and~\eqref{poly-alpha-beta-best-ineq} has been applied in the proof of~\cite[Lemma~3]{Open-TJM-2003.tex} to prove that the inequality
\begin{equation}\label{gamma(t/(1+2t))}
\frac{1+2t}{2t^2}\biggl[\ln\Gamma\biggl(\frac{t}{1+2t}\biggr)-\ln\Gamma(t)\biggr]<1-\psi(t)
\end{equation}
is valid for $t>0$.
\end{rem}

\section{Proofs of Theorem~\ref{sharp-ineq-polygamma-thm} and corollaries}

Now we are in a position to prove Theorem~\ref{sharp-ineq-polygamma-thm} and the above corollaries.

\begin{proof}[Proof of Theorem~\ref{sharp-ineq-polygamma-thm}]
For $x>0$ and $k\in\mathbb{N}$, let
\begin{equation}\label{poly-alpha-beta-best-ineq-rew}
h_k(x)=\Biggl[\frac{(k-1)!}{\bigl\vert\psi^{(k)}(x)\bigr\vert-\frac{k!}{x^{k+1}}}\Biggr]^{1/k}-x.
\end{equation}
Using the right-hand side inequality in~\eqref{pai-alzer-thm9-ineq} for $n\ge0$ yields
\begin{align*}
h_k(x)&>\Biggl[\frac{(k-1)!}{S_k(2n+1;x)-\frac{k!}{x^{k+1}}}\Biggr]^{1/k}-x\\
&=\Biggl\{\frac{(k-1)!}{\frac{(k-1)!}{x^k}-\frac{k!}{2x^{k+1}} +\sum_{i=1}^{2n+1}B_{2i}\bigl[\prod_{j=1}^{k-1}(2i+j)\bigr] \frac1{x^{2i+k}}}\Biggr\}^{1/k}-x\\
&=x\Biggl\{\Biggl[\frac1{1-\frac{k}{2x} +\sum_{i=1}^{2n+1}\frac{B_{2i}}{(k-1)!}\bigl[\prod_{j=1}^{k-1}(2i+j)\bigr] \frac1{x^{2i}}}\Biggr]^{1/k}-1\Biggr\}\\
&=\frac1{u}\Biggl\{\Biggl[\frac1{1-\frac{k}{2}u +\sum_{i=1}^{2n+1}\frac{B_{2i}}{(k-1)!}\bigl[\prod_{j=1}^{k-1}(2i+j)\bigr] u^{2i}}\Biggr]^{1/k}-1\Biggr\},\\
&\to\frac12\quad \text{as $u\to0^+$, or say, $x\to\infty$}.
\end{align*}
Similarly, making use of the left-hand side inequality in~\eqref{pai-alzer-thm9-ineq} for $n\ge0$ results in
\begin{align*}
h_k(x)&<\Biggl[\frac{(k-1)!}{S_k(2n;x)-\frac{k!}{x^{k+1}}}\Biggr]^{1/k}-x\\*
&=\Biggl\{\frac{(k-1)!}{\frac{(k-1)!}{x^k}-\frac{k!}{2x^{k+1}} +\sum_{i=1}^{2n}B_{2i}\bigl[\prod_{j=1}^{k-1}(2i+j)\bigr] \frac1{x^{2i+k}}}\Biggr\}^{1/k}-x\\
&=x\Biggl\{\Biggl[\frac1{1-\frac{k}{2x} +\sum_{i=1}^{2n}\frac{B_{2i}}{(k-1)!}\bigl[\prod_{j=1}^{k-1}(2i+j)\bigr] \frac1{x^{2i}}}\Biggr]^{1/k}-1\Biggr\}\\
&=\frac1{u}\Biggl\{\Biggl[\frac1{1-\frac{k}{2}u +\sum_{i=1}^{2n}\frac{B_{2i}}{(k-1)!}\bigl[\prod_{j=1}^{k-1}(2i+j)\bigr] u^{2i}}\Biggr]^{1/k}-1\Biggr\},\\
&\to\frac12\quad \text{as $u\to0^+$, or say, $x\to\infty$}.
\end{align*}
In a word, it follows that
\begin{equation}\label{h_k(x)=frac12}
\lim_{x\to\infty}h_k(x)=\frac12.
\end{equation}
\par
By the well-known recurrence formula~\cite[p.~260, 6.4.6]{abram}
\begin{equation}\label{psisymp4}
\psi^{(n-1)}(x+1)=\psi^{(n-1)}(x)+\frac{(-1)^{n-1}(n-1)!}{x^n}
\end{equation}
and the integral representation~\cite[p.~260, 6.4.1]{abram}
\begin{equation}\label{psim}
\psi ^{(k)}(x)=(-1)^{k+1}\int_{0}^{\infty}\frac{t^{k}}{1-e^{-t}}e^{-xt}\td t
\end{equation}
for $x>0$ and $n\in\mathbb{N}$, we have
\begin{equation}
\lim_{x\to0^+}h_k(x)
=\lim_{x\to0^+}\biggl[\frac{(k-1)!}{\vert\psi^{(k)}(x+1)\vert}\biggr]^{1/k}
=\biggl[\frac{(k-1)!}{\vert\psi^{(k)}(1)\vert}\biggr]^{1/k}.\label{psi{(k)}(1)}
\end{equation}
\par
By virtue of~\eqref{psisymp4} and~\eqref{psim}, straightforward computation yields
\begin{align*}
h_k'(x)&=\frac{\td}{\td x}\Biggl\{\biggl[\frac{(k-1)!}{(-1)^{k+1}\psi^{(k)}(x+1)}\biggr]^{1/k} -x\Biggr\}\\
&=-\frac{\psi^{(k+1)}(x+1)}{k\psi^{(k)}(x+1)} \biggl[\frac{(k-1)!}{(-1)^{k+1}\psi^{(k)}(x+1)}\biggr]^{1/k}-1\\
&=\frac{\bigl\vert\psi^{(k+1)}(x+1)\bigr\vert}{k!} \Biggl[\frac{(k-1)!}{\bigl\vert\psi^{(k)}(x+1)\bigr\vert}\Biggr]^{1+1/k}-1\\
&=\Biggl[\sqrt[k+1]{\frac{\bigl\vert\psi^{(k+1)}(x+1)\bigr\vert}{k!}}\, \sqrt[k]{\frac{(k-1)!}{\bigl\vert\psi^{(k)}(x+1)\bigr\vert}}\,\Biggr]^{k+1}-1.
\end{align*}
By virtue of the inequality~\eqref{gen-1-1}, it follows that $h_k'(x)<0$ on $(0,\infty)$, which means that the functions $h_k(x)$ for $k\in\mathbb{N}$ are strictly decreasing on $(0,\infty)$.
\par
In conclusion, from a combination of the decreasing monotonicity of $h_k(x)$ with the limits~\eqref{h_k(x)=frac12} and~\eqref{psi{(k)}(1)}, Theorem~\ref{sharp-ineq-polygamma-thm} follows immediately.
\end{proof}

\begin{proof}[Proof of Corollary~\ref{sharp-ineq-polygamma-cor-1}]
This follows from
\begin{equation*}
\bigl\vert\psi^{(k)}(x)\bigr\vert-\frac{k!}{x^{k+1}}=\bigl\vert\psi^{(k)}(x+1)\bigr\vert,
\end{equation*}
an equivalence of~\eqref{psisymp4}, for $k\in\mathbb{N}$ and $x>0$.
\end{proof}

\begin{proof}[Proof of Corollary~\ref{sharp-ineq-polygamma-cor-2}]
Utilizing the identity~\eqref{psisymp4} and the integral express~\eqref{psim} shows
\begin{equation*}
\bigl\vert\psi^{(k)}(x+m)\bigr\vert=\bigl\vert\psi^{(k)}(x)\bigr\vert-\sum_{i=1}^m\frac{k!}{(x+i-1)^{k+1}}
\end{equation*}
for $k,m\in\mathbb{N}$ and $x>0$. Combining this with~\eqref{poly-alpha-beta-best-ineq} and the case $x+m$ of~\eqref{poly-alpha-beta-best-ineq} respectively leads to the inequalities~\eqref{poly-alpha-beta-best-ineq-c-3} and~\eqref{poly-alpha-beta-best-ineq-c-2}. Corollary~\ref{sharp-ineq-polygamma-cor-2} is proved.
\end{proof}


\begin{thebibliography}{99}

\bibitem{abram}
M. Abramowitz and I. A. Stegun (Eds), \textit{Handbook of Mathematical Functions with Formulas, Graphs, and Mathematical Tables}, National Bureau of Standards, Applied Mathematics Series \textbf{55}, 4th printing, with corrections, Washington, 1965.

\bibitem{Allasia-Gior-Pecaric-MIA-02}
G. Allasia, C. Giordano and J. Pe\v{c}ari\'c, \textit{Inequalities for the gamma function relating to asymptotic expasions}, Math. Inequal. Appl. \textbf{5} (2002), no.~3, 543\nobreakdash--555.

\bibitem{psi-alzer}
H. Alzer, \textit{On some inequalities for the gamma and psi functions}, Math. Comp. \textbf{66} (1997), no.~217, 373\nobreakdash--389.

\bibitem{forum-alzer}
H. Alzer, \textit{Sharp inequalities for the digamma and polygamma functions}, Forum Math. \textbf{16} (2004), 181\nobreakdash--221.

\bibitem{alzer-grinshpan}
H. Alzer and A. Z. Grinshpan, \textit{Inequalities for the gamma and $q$-gamma functions}, J. Approx. Theory \textbf{144} (2007), 67\nobreakdash--83.

\bibitem{Atanassov}
R. D. Atanassov and U. V. Tsoukrovski, \textit{Some properties of a class of logarithmically completely monotonic functions}, C. R. Acad. Bulgare Sci. \textbf{41} (1988), no.~2, 21\nobreakdash--23.

\bibitem{batir-interest}
N. Batir, \textit{An interesting double inequality for Euler's gamma function}, J. Inequal. Pure Appl. Math. \textbf{5} (2004), no.~4, Art.~97; Available online at \url{http://jipam.vu.edu.au/article.php?sid=452}.

\bibitem{Batir-Arch-Math-08}
N. Batir, \textit{Inequalities for the gamma function}, Arch. Math. \textbf{91} (2008), 554\nobreakdash--563.

\bibitem{batir-jmaa-06-05-065}
N. Batir, \textit{On some properties of digamma and polygamma functions}, J. Math. Anal. Appl. \textbf{328} (2007), no.~1, 452\nobreakdash--465; Available online at \url{http://dx.doi.org/10.1016/j.jmaa.2006.05.065}.

\bibitem{batir-new}
N. Batir, \textit{Some new inequalities for gamma and polygamma functions}, J. Inequal. Pure Appl. Math. \textbf{6} (2005), no.~4, Art.~103; Available online at \url{http://jipam.vu.edu.au/article.php?sid=577}.

\bibitem{batir-new-RGMIA}
N. Batir, \textit{Some new inequalities for gamma and polygamma functions}, RGMIA Res. Rep. Coll. \textbf{7} (2004), no.~3, Art.~1; Available online at \url{http://www.staff.vu.edu.au/rgmia/v7n3.asp}.

\bibitem{CBerg}
C. Berg, \textit{Integral representation of some functions related to the gamma function}, Mediterr. J. Math. \textbf{1} (2004), no.~4, 433\nobreakdash--439.

\bibitem{chen-qi-log-jmaa}
Ch.-P. Chen and F. Qi, \textit{Logarithmically completely monotonic functions relating to the gamma function}, J. Math. Anal. Appl. \textbf{321} (2006), no.~1, 405\nobreakdash--411.

\bibitem{egp}
N. Elezovi\'c, C. Giordano and J. Pe\v{c}ari\'c, \textit{The best bounds in Gautschi's inequality}, Math. Inequal. Appl. \textbf{3} (2000), 239\nobreakdash--252.

\bibitem{Gao-0709.1126v2-Arxiv}
P. Gao, \textit{Some monotonicity properties of gamma and $q$-gamma functions}, Available onlie at \url{http://arxiv.org/abs/0709.1126v2}.

\bibitem{grin-ismail}
A. Z. Grinshpan and M. E. H. Ismail, \textit{Completely monotonic functions involving the gamma and $q$\nobreakdash-gamma functions}, Proc. Amer. Math. Soc. \textbf{134} (2006), 1153\nobreakdash--1160.

\bibitem{subadditive-qi.tex}
B.-N. Guo, R.-J. Chen and F. Qi, \textit{A class of completely monotonic functions involving the polygamma functions}, J. Math. Anal. Approx. Theory \textbf{1} (2006), no.~2, 124\nobreakdash--134.

\bibitem{Guo-Qi-12-07.tex}
S. Guo and F. Qi, \textit{More supplements to a class of logarithmically completely monotonic functions associated with the gamma function}, submitted.

\bibitem{Guo-Qi-Srivastava2007.tex}
S. Guo, F. Qi and H. M. Srivastava, \textit{Necessary and sufficient conditions for two classes of functions to be logarithmically completely monotonic}, Integral Transforms Spec. Funct. \textbf{18} (2007), no.~11, 819\nobreakdash--826.

\bibitem{Guo-Qi-Srivastava2007-02.tex}
S. Guo, F. Qi and H. M. Srivastava, \textit{Supplements to a class of logarithmically completely monotonic functions associated with the gamma function}, Appl. Math. Comput. \textbf{197} (2008), no.~2, 768\nobreakdash--774.

\bibitem{Guo-Srivastava-AML-08}
S. Guo and H. M. Srivastava, \textit{A class of logarithmically completely monotonic functions}, Appl. Math. Lett. \textbf{21} (2008), no.~11, 1134\nobreakdash--1141.

\bibitem{horn}
R. A. Horn, \textit{On infinitely divisible matrices, kernels and functions}, \textit{Z. Wahrscheinlichkeitstheorie und Verw. Geb}, \textbf{8} (1967), 219\nobreakdash--230.

\bibitem{Ismail-Muldoon-119}
M. E. H. Ismail and M. E. Muldoon, \textit{Inequalities and monotonicity properties for gamma and $q$-gamma functions}, in: R.V.M. Zahar (Ed.), Approximation and Computation: A Festschrift in Honour of Walter Gautschi, ISNM, Vol. \textbf{119}, BirkhRauser, Basel, 1994, 309\nobreakdash--323.

\bibitem{Kazarinoff-56}
D. K. Kazarinoff, \textit{On Wallis' formula}, Edinburgh Math. Notes \textbf{1956} (1956), no.~40, 19\nobreakdash--21.

\bibitem{Koumandos-jmaa-06}
S. Koumandos, \textit{Remarks on some completely monotonic functions}, J. Math. Anal. Appl. \textbf{324} (2006), no.~2, 1458\nobreakdash--1461.

\bibitem{merkle-jmaa-96}
M. Merkle, \textit{Logarithmic convexity and inequalities for the gamma function}, J. Math. Anal. Appl.  \textbf{203} (1996), no~2, 369\nobreakdash--380.

\bibitem{mpf-93}
D. S. Mitrinovi\'c,  J. E. Pe\v{c}ari\'c and A. M. Fink, \textit{Classical and New Inequalities in Analysis}, Kluwer Academic Publishers, Dordrecht/Boston/London, 1993.

\bibitem{Muldoon-78}
M. E. Muldoon, \textit{Some monotonicity properties and characterizations of the gamma function}, Aequationes Math. \textbf{18} (1978), 54\nobreakdash--63.

\bibitem{notes-best-simple-equiv.tex-RGMIA}
F. Qi, \textit{A completely monotonic function involving divided difference of psi function and an equivalent inequality involving sum}, RGMIA Res. Rep. Coll. \textbf{9} (2006), no.~4, Art.~5; Available online at \url{http://www.staff.vu.edu.au/rgmia/v9n4.asp}.

\bibitem{notes-best-simple-equiv.tex}
F. Qi, \textit{A completely monotonic function involving the divided difference of the psi function and an equivalent inequality involving sums}, ANZIAM J. \textbf{48} (2007), no.~4, 523\nobreakdash--532.

\bibitem{notes-best-simple-open}
F. Qi, \textit{A completely monotonic function involving divided differences of psi and polygamma functions and an application}, RGMIA Res. Rep. Coll. \textbf{9} (2006), no.~4, Art.~8; Available online at \url{http://www.staff.vu.edu.au/rgmia/v9n4.asp}.

\bibitem{bounds-two-gammas.tex}
F. Qi, \textit{Bounds for the ratio of two gamma functions}, RGMIA Res. Rep. Coll. \textbf{11} (2008), no.~3, Art.~1; Available online at \url{http://www.staff.vu.edu.au/rgmia/v11n3.asp}.

\bibitem{Wendel-Gautschi-type-ineq.tex}
F. Qi, \textit{Bounds for the ratio of two gamma functions---From Wendel's and related inequalities to logarithmically completely monotonic functions}, submitted.

\bibitem{Wendel2Elezovic.tex}
F. Qi, \textit{Bounds for the ratio of two gamma functions---From Wendel's limit to Elezovi\'c-Giordano-Pe\v{c}ari\'c's theorem}, Available online at \url{http://arxiv.org/abs/0902.2514}.

\bibitem{simple-equiv.tex}
F. Qi, \textit{Complete monotonicity results of a function involving the divided difference of two psi functions and consequences}, submitted.

\bibitem{Comp-Mon-Digamma-Trigamma-Divided.tex}
F. Qi, \textit{Necessary and sufficient conditions for a function involving divided differences of the di- and tri-gamma functions to be completely monotonic}, submitted.

\bibitem{notes-best-simple.tex-rgmia}
F. Qi, \textit{The best bounds in Kershaw's inequality and two completely monotonic functions}, RGMIA Res. Rep. Coll. \textbf{9} (2006), no.~4, Art.~2; Available online at \url{http://www.staff.vu.edu.au/rgmia/v9n4.asp}.

\bibitem{sandor-gamma-2-ITSF.tex}
F. Qi, \textit{Three classes of logarithmically completely monotonic functions involving gamma and psi functions}, Integral Transforms Spec. Funct. \textbf{18} (2007), no.~7, 503\nobreakdash--509.

\bibitem{sandor-gamma-2-ITSF.tex-rgmia}
F. Qi, \textit{Three classes of logarithmically completely monotonic functions involving gamma and psi functions}, RGMIA Res. Rep. Coll. \textbf{9} (2006), Suppl., Art.~6; Available online at \url{http://www.staff.vu.edu.au/rgmia/v9(E).asp}.

\bibitem{simple-equiv-simple-rev.tex}
F. Qi, P. Cerone and S. S. Dragomir, \textit{Complete monotonicity results of divided difference of psi functions and new bounds for ratio of two gamma functions}, submitted.

\bibitem{compmon2}
F. Qi and Ch.-P. Chen, \textit{A complete monotonicity property of the gamma function}, J. Math. Anal. Appl. \textbf{296} (2004), no.~2, 603\nobreakdash--607.

\bibitem{poly-comp-mon.tex}
F. Qi, R.-Q. Cui, Ch.-P. Chen, and B.-N. Guo, \textit{Some completely monotonic functions involving polygamma functions and an application}, J. Math. Anal. Appl. \textbf{310} (2005), no.~1, 303\nobreakdash--308.

\bibitem{notes-best-simple-open.tex-rev}
F. Qi and B.-N. Guo, \textit{A class of completely monotonic functions involving divided differences of psi and polygamma functions and some applications}, Available online at \url{http://arxiv.org/abs/0903.1430}.

\bibitem{Extension-TJM-2003.tex}
F. Qi and B.-N. Guo, \textit{A logarithmically completely monotonic function involving the gamma function}, Taiwanese J. Math. \textbf{14} (2010), no.~2, in press.

\bibitem{theta-new-proof.tex}
F. Qi and B.-N. Guo, \textit{A new proof of complete monotonicity of a function involving psi function}, RGMIA Res. Rep. Coll. \textbf{11} (2008), no.~3, Art.~12; Available online at \url{http://www.staff.vu.edu.au/rgmia/v11n3.asp}.

\bibitem{absolute-mon.tex}
F. Qi and B.-N. Guo, \textit{A property of logarithmically absolutely monotonic functions and the logarithmically complete monotonicity of a power-exponential function}, submitted.

\bibitem{minus-one}
F. Qi and B.-N. Guo, \textit{Complete monotonicities of functions involving the gamma and digamma functions}, RGMIA Res. Rep. Coll. \textbf{7} (2004), no.~1, Art.~8, 63\nobreakdash--72; Available online at
\url{http://www.staff.vu.edu.au/rgmia/v7n1.asp}.

\bibitem{polygamma-square-polygamma.tex}
F. Qi and B.-N. Guo, \textit{Complete monotonicity of a polygamma function plus the square of another polygamma function}, submitted.

\bibitem{Open-TJM-2003.tex}
F. Qi and B.-N. Guo, \textit{Necessary and sufficient conditions for a function involving a ratio of gamma functions to be logarithmically completely monotonic}, submitted.

\bibitem{AAM-Qi-09-PolyGamma.tex}
F. Qi and B.-N. Guo, \textit{Necessary and sufficient conditions for functions involving the tri- and tetra-gamma functions to be completely monotonic}, Adv. Appl. Math. (2009), in press.

\bibitem{Infinite-family-Digamma.tex}
F. Qi and B.-N. Guo, \textit{Sharp inequalities for the psi function and harmonic numbers}, Available online at \url{http://arxiv.org/abs/0902.2524}.

\bibitem{property-psi.tex}
F. Qi and B.-N. Guo, \textit{Some properties of the psi and polygamma functions}, Available online at \url{http://arxiv.org/abs/0903.1003}.

\bibitem{notes-best-simple-rev.tex}
F. Qi and B.-N. Guo, \textit{Completely monotonic functions involving divided differences of the di- and tri-gamma functions and some applications}, Commun. Pure Appl. Anal. (2009), in press.

\bibitem{e-gam-rat-comp-mon}
F. Qi, B.-N. Guo and Ch.-P. Chen, \textit{Some completely monotonic functions involving the gamma and polygamma functions}, J. Aust. Math. Soc. \textbf{80} (2006), 81\nobreakdash--88.

\bibitem{auscm-rgmia}
F. Qi, B.-N. Guo and Ch.-P. Chen, \textit{Some completely monotonic functions involving the gamma and polygamma functions}, RGMIA Res. Rep. Coll. \textbf{7} (2004), no.~1, Art.~5, 31\nobreakdash--36; Available online at \url{http://www.staff.vu.edu.au/rgmia/v7n1.asp}.

\bibitem{subadditive-qi-guo.tex}
F. Qi, S. Guo and B.-N. Guo, \textit{Complete monotonicity of some functions involving polygamma functions}, submitted.

\bibitem{subadditive-qi-3.tex}
F. Qi, S. Guo and B.-N. Guo, \textit{Note on a class of completely monotonic functions involving the polygamma functions}, RGMIA Res. Rep. Coll. \textbf{10} (2007), no.~1, Art.~5; Available online at \url{http://www.staff.vu.edu.au/rgmia/v10n1.asp}.

\bibitem{schur-complete}
F. Qi, W. Li and B.-N. Guo, \textit{Generalizations of a theorem of I. Schur}, RGMIA Res. Rep. Coll. \textbf{9} (2006), no.~3, Art.~15; Available online at \url{http://www.staff.vu.edu.au/rgmia/v9n3.asp}. B\`ud\v{e}ngsh\`i Y\=anji\=u T\=ongx\`un (Communications in Studies on Inequalities) \textbf{13} (2006), no.~4, 355\nobreakdash--364.

\bibitem{widder}
D. V. Widder, \textit{The Laplace Transform}, Princeton University Press, Princeton, 1941.

\end{thebibliography}
\end{document}